\documentclass[10pt,reqno]{amsart}
\usepackage{amsmath}
\usepackage{amssymb}
\usepackage{mathrsfs}
\usepackage{xcolor, enumitem, comment, todonotes}
\usepackage[all]{xy}
 \usepackage{url}
 \usepackage{amsthm}
 \usepackage{thmtools}
\usepackage{thm-restate}
\usepackage{hyperref,soul}
\usepackage{caption}
\usepackage{tikz-3dplot}

\usepackage[T1]{fontenc}

\theoremstyle{plain}
\newtheorem*{theorem*}{Theorem}

\newtheorem{theorem}{Theorem}[section]
\newtheorem{claim}[theorem]{Claim}

\newtheorem{lemma}[theorem]{Lemma}

\newtheorem{proposition}[theorem]{Proposition}
\newtheorem{corollary}[theorem]{Corollary}

\theoremstyle{definition}
\newtheorem{definition}[theorem]{Definition}

\newtheorem{question}[theorem]{Question}

\theoremstyle{remark}

\newcount\skewfactor
\def\mathunderaccent#1#2 {\let\theaccent#1\skewfactor#2
\mathpalette\putaccentunder}
\def\putaccentunder#1#2{\oalign{$#1#2$\crcr\hidewidth
\vbox to.2ex{\hbox{$#1\skew\skewfactor\theaccent{}$}\vss}\hidewidth}}

\def\smallbox#1{\leavevmode\thinspace\hbox{\vrule\vtop{\vbox
   {\hrule\kern1pt\hbox{\vphantom{\tt/}\thinspace{\tt#1}\thinspace}}
   \kern1pt\hrule}\vrule}\thinspace}

\DeclareMathOperator{\im}{Im}

\title{Fundamental groups and descriptive set theory}
\author{Fanxin Wu}

\address{Department of Mathematical Sciences, Rutgers University, Piscataway, NJ 08854}
\email{fw173@rutgers.edu}

\date{\today}
\subjclass[2020]{03E15, 54H05, 20F34} 
\keywords{analytic equivalence relations, fundamental group, Hawaiian earring}

\begin{document}

\maketitle

\begin{abstract}
We study the homotopy of loops in a fixed path-connected Polish space from a descriptive set-theoretic viewpoint. We show that many analytic equivalence relations arise this way, and many do not. We also study the ``free group'' over an equivalence relation.
\end{abstract}

\section{Introduction}

Descriptive set theory provides explanations to many phenomena in mathematics: why certain characterization problems are difficult \cite{becker1992descriptive}, why certain classification problems are difficult \cite{kechris1999new,gao2008invariant,thomas2003classification,kulikov2017non}, and why certain constructions cannot be made uniform \cite{thomas2011descriptive}. In this paper we explore the possibility of using descriptive set theory to explain why the fundamental groups of certain spaces are hard to describe.

Let $\mathbb{E}$ be the Hawaiian earring space. Its fundamental group is uncountable, and often considered too wild to describe explicitly. However, there does in fact exist a fairly concrete description of $\pi(\mathbb{E})$ as a certain subgroup of an inverse limit of free groups \cite{morgan1986van}. Translating this result into the language of equivalence relations, we see that the homotopy relation on the loop space of $\mathbb{E}$ is smooth. This motivates us to study the Borel complexity of homotopy relation on the loop space $\Omega X$ for a general path-connected Polish space $X$, which in some sense measures how hard it is to describe the fundamental group $\pi(X)$.

Our notations are mostly standard. For general background on descriptive set theory we refer to \cite{kechris2012classical,gao2008invariant}. The organization of the paper is as follows. In Section 1 we define the main object of interest, and discuss the example of the earring space $\mathbb{E}$. Section 2 contains limiting results on which analytic equivalence relations can be realized as homotopy of loops. Section 3 defines an operation $E\mapsto F(E)$ on analytic equivalence relations, and shows that $F(E)$ can always be realized as homotopy, building on a similar result of Becker concerning path-components. Section 4 gives some examples of $E$ that are bireducible with $F(E)$.

\begin{center}
\textbf{Acknowledgement}
\end{center}

I am grateful to Jeremy Brazas, Alexander Kechris and Simon Thomas for valuable discussions. After a draft of this paper was finished, I learned that some of the results also appeared in unpublished notes by Orestis Raptis. I thank Alexander Kechris for kindly sharing the notes.

\section{Homotopy of loops}

Suppose $X$ is Polish and path-connected. Fix any basepoint $x_0\in X$ and consider the loop space

\begin{center}
$\Omega X=\{f\in C([0,1],X):f(0)=f(1)=x_0\}$,
\end{center} 

which is closed in $C([0,1],X)$ and thus Polish. We consider the homotopy relation $\sim_X$ on $\Omega X$, namely two loops $f,g$ are homotopic if and only if there exists a continuous $F:[0,1]\times[0,1]\rightarrow X$ such that $F(\cdot,0)=f$, $F(\cdot,1)=g$, and $f(0,t)=f(1,t)=x_0$. Clearly this is an analytic equivalence relation, and the quotient set is the fundamental group $\pi(X)$. The usual proof that $\pi(X)$ does not depend on basepoint shows that the Borel complexity of $\sim_X$ does not depend on basepoint either.

It seems the first and only published result concerning $\sim_X$ is due to Shelah \cite{shelah1988can}.

\begin{theorem*}[Shelah]
Suppose $X$ is a compact Polish space that is path-connected and locally path-connected. If $X$ is semi-locally simply-connected (slsc), then $\pi(X)$ is finitely generated; if $X$ is not slsc, then $\Omega X$ has a perfect set of non-homotopic loops.
\end{theorem*}

Note that if $\pi(X)$ is finitely generated then in particular it is countable, so $\sim_X$ has countably many classes. In the terminology of Borel reducibility, Shelah's theorem implies that for $X$ as above, either $\sim_X\leq_B\mathrm{id}(\omega)$ or $\mathrm{id}({}^\omega2)\sqsubseteq_c \sim_X$ (compare with Burgess trichotomy).

The first alternative in Shelah's theorem has been strengthened \cite{cannon2006fundamental,dydak2011alternate}: if $X$ is slsc then $\pi(X)$ is in fact finitely presented. The proof of the second alternative roughly proceeds as follows: if $X$ is not slsc, then around some point there are smaller and smaller non-trivial loops, so $X$ locally looks like the Hawaiian earring space $\mathbb{E}$, and one can explicitly construct many non-homotopic loops.

The Hawaiian earring space $\mathbb{E}$ is the union of the circles $(x-\frac{1}{n})^2+y^2=(\frac{1}{n})^2$ for $n\geq 1$, and is a standard example of a space that is not slsc. An alternative definition is as follows. Let $\mathbb{E}_n=\bigvee_{k\leq n}S^1$ be the wedge sum of $n$ circles, so $\mathbb{E}_{n+1}$ naturally maps onto $\mathbb{E}_n$ by collapsing the $n+1$-th circle to the base point. $\mathbb{E}$ can be identified with the inverse limit $\varprojlim_n \mathbb{E}_n$.

The ``earring group'' $\pi(\mathbb{E})$ is often considered exotic, but nevertheless admits the following concrete description. Universal property of inverse limit gives rise to a natural map $\varphi:\pi(\mathbb{E})\rightarrow\varprojlim_n \pi(\mathbb{E}_n)=\varprojlim_n F_n$, where $F_n$ is the free group on $n$ generators $\{g_1,\dots,g_n\}$, and the map from $F_{n+1}$ to $F_n$ simply deletes all appearances of $g_{n+1}$ and reduces the word. Using nontrivial results from continuum theory and shape theory, it can be shown that $\varphi$ is injective \cite{morgan1986van}; see also \cite{cannon2000combinatorial}. This is a special case of the more general phenomenon of \textit{shape injectivity}. Moreover, the image of $\varphi$ is quite simple: it is the subset of $\varprojlim_n F_n$ consisting of all infinite sequences $(w_n:n<\infty)$ of reduced words in the $g_n$'s such that for each $k$, the number of times $g_k$ appears in $w_n$ is eventually constant as $n\rightarrow\infty$ (note that it is non-decreasing).

In descriptive set-theoretic terms, the above discussion shows that $\sim_\mathbb{E}$ is not only smooth, but \textit{faithfully smooth}. The following definition is taken from Gao \cite{gao2008invariant}.

\begin{definition}
Let $X,Y$ be standard Borel spaces and $E,F$ be equivalence relations on $X,Y$, respectively. A Borel reduction $f:X\rightarrow Y$ is \textit{faithful} if for any $E$-invariant Borel $A\subseteq X$, the $F$-saturation of $f(A)$ is Borel. $E$ is \textit{faithfully smooth} if there exists a faithful Borel reduction to $\mathrm{id}_Z$ for some Polish space $Z$.
\end{definition}

\begin{proposition}\label{faithfully_smooth}
Let $X$ be a standard Borel space and $E$ be an equivalence relation on $E$.

(i) A Borel reduction $f:X\rightarrow\mathbb{R}$ of $E$ to $\mathrm{id}_Y$ is faithful if and only if $f(X)$ is Borel.

(ii) If $E$ is faithfully smooth, then every reduction of $E$ to some $\mathrm{id}_Y$ is faithful. 
\end{proposition}
\begin{proof}
(i) The forward direction is clear. If $f(X)$ is Borel and $A\subseteq X$ is $E$-invariant, then $f(X)$ is the disjoint union of $f(A)$ and $f(X\setminus A)$, both analytic, so by Suslin's theorem $f(A)$ is Borel.

(ii) Suppose $f$ is a faithful reduction of $E$ to $\mathrm{id}_Y$ and $g$ is a reduction of $E$ to some $\mathrm{id}_Z$. There is a natural map $h:Y\supseteq f(X)\rightarrow Z$ sending $f(x)$ to $g(x)$, so $h\circ f=g$. Note that $h$ is Borel: if $B\subseteq Z$ is Borel and $y=f(x)$, then $h(y)\in B$ iff $g(x)\in B$, which shows $h^{-1}(B)$ is Borel. Clearly $h$ is injective, so $g(X)=h(f(X))$ is Borel.
\end{proof}

\begin{proposition}
$\sim_{\mathbb{E}}$ is faithfully smooth.
\end{proposition}
\begin{proof}
Let $\psi_n$ be the natural map from $\Omega\mathbb{E}$ to $\pi(\mathbb{E}_n)=F_n$; each fiber is analytic, and thus Borel since $F_n$ is countable. Let $\psi$ be the induced map from $\Omega\mathbb{E}$ to the Polish space $\varprojlim_n F_n$. As mentioned above, the induced map $\varphi:\pi(\mathbb{E})\rightarrow\varprojlim_n F_n$ is injective, and the image consists of those sequences $(\dots,w_3,w_2,w_1)$ of coherent reduced words satisfying the following: for each $k$, the number of times $g_k$ appears in $w_n$ is eventually constant as $n\rightarrow\infty$. This shows $\im(\psi)$ is $\mathbf{\Pi^0_3}$, in particular Borel. Thus $\psi$ is a faithful reduction by Proposition \ref{faithfully_smooth}.
\end{proof}

More generally, $\sim_X$ is smooth whenever shape injectivity holds, e.g., when $X$ is one-dimensional \cite{eda1998fundamental} or planar \cite{fischer2005fundamental}, in which case $\sim_X$ is faithfully smooth if and only if the shape image is Borel.

\section{Obstruction}

In the next section we show that a rich class of analytic equivalence relations can be realized as $\sim_X$ for suitable $X$. Here we show that there are also analytic equivalence relations that cannot be realized this way.

\begin{lemma}\label{obstruction}
Suppose $X$ is a standard Borel space and $E$ is an analytic equivalence relation on $X$ with at least two equivalence classes, such that for any two Borel reductions $h_1,h_2$ of $E$ to itself there is $x\in X$ for which $h_1(x)Eh_2(x)$, then $E$ is not Borel bireducible with $\sim_Y$ for any path-connected Polish space $Y$.
\end{lemma}
\begin{proof}
We prove the contrapositive. Suppose $E$ and $\sim_Y$ are bireducible as witnessed by $f:X\rightarrow\Omega Y$ and $g:\Omega Y\rightarrow X$. Since $E$ has at least two classes, so does $\sim_Y$. For each $p\in\Omega Y$, the map $h_p:X\rightarrow X,\ x\mapsto g(p\cdot f(x))$ is a self-reduction of $E$, where $p\cdot f(x)$ denotes concatenation in $\Omega Y$. Clearly $\lnot h_{p_1}(x)Eh_{p_2}(x)$ for non-homotopic $p_1,p_2$.
\end{proof}

\begin{corollary}
If $E$ is an analytic equivalence relation on a standard Borel space $X$ with exactly one non-Borel class, then $E$ is not Borel bireducible with $\sim_Y$ for any path-connected Polish space $Y$.
\end{corollary}
\begin{proof}
Any self-reduction of $E$ must send the unique non-Borel class to itself.
\end{proof}

The standard analytic equivalence relation with exactly one non-Borel class is $\mathrm{id}_{\omega_1}$. It is the analytic equivalence relation on $LO$, the Polish space of countable linear orders, where two linear orders are equivalent if either they are both ill-founded or they are isomorphic. This example is a bit artificial, since for general analytic equivalence relations (as opposed to Borel or orbit equivalence relations) one often consider the more general notion of $\Delta^1_2$ reduction, and it is not unreasonable to expect $\mathrm{id}_{\omega_1}$ to be $\Delta^1_2$ bireducible with $\sim_Y$ for some $Y$.

Let $\mathrm{id}_\mathbb{R}\sqcup\mathrm{id}_{\omega_1}$ be the disjoint union of $\mathrm{id}_\mathbb{R}$ and $\mathrm{id}_{\omega_1}$ defined on $\mathbb{R}\sqcup LO$. We show that $\mathrm{id}_\mathbb{R}\sqcup\mathrm{id}_{\omega_1}$ is not bireducible with any $\sim_Y$ in any reasonable sense: Borel, $\Delta^1_2$, projective, etc., at least assuming some large cardinal.

\begin{proposition}
If all subsets of $\mathbb{R}$ are Lebesgue measurable, then $\mathbb{R}\sqcup\omega_1$ does not admit a group structure, so $\mathrm{id}_\mathbb{R}\sqcup\mathrm{id}_{\omega_1}$ is not bireducible with any $\sim_Y$. In particular this is true in $L(\mathbb{R})$ under large cardinals.
\end{proposition}
\begin{proof}
It is well known that ``all subsets of $\mathbb{R}$ are Lebesgue measurable'' implies there is no injection between $\mathbb{R}$ and $\omega_1$, which implies $\mathbb{R}\sqcup\omega_1$ is not a group by a standard argument; see \cite{hajnal1972some}. For completeness we include the proof. Suppose for contradiction that $(\mathbb{R}\sqcup\omega_1,\cdot)$ is a group. Then for any $r\in\mathbb{R}$ there exists an $\alpha$ such that $r\cdot\alpha\in\omega_1$, since otherwise $\alpha\mapsto r\cdot\alpha$ is an injection of $\omega_1$ into $\mathbb{R}$. Then we can define an injection from $\mathbb{R}$ into $\omega_1\times\omega_1$, and thus into $\omega_1$, by sending $r$ to $(\alpha,r\cdot\alpha)$ where $\alpha$ is the least such that $r\cdot\alpha\in\omega_1$; this is an injection because of cancelltion law.

If $\mathrm{id}_\mathbb{R}\sqcup\mathrm{id}_{\omega_1}$ were bireducible with $\sim_Y$, then by Cantor--Bernstein there would be a bijection between the quotient sets, namely $\mathbb{R}\sqcup\omega_1$ and $\pi(Y)$, which makes $\mathbb{R}\sqcup\omega_1$ a group, a contradiction.
\end{proof}

Simon Thomas has pointed out a $\mathsf{ZFC}$ example of a countable Borel equivalence relation $E$ with an ergodic measure, such that for any self reduction $f$ of $E$, we have $f(x)E x$ on a measure one set; in particular Lemma \ref{obstruction} applies. See \cite[Theorem 4.2]{gefter1996outer}, \cite[Lemma 3.4]{thomas2002some} and \cite[Theorem 4.4]{thomas2003superrigidity}.

\section{Realization}

To realize analytic equivalence relations as homotopy, we first recall a result of Becker: any analytic equivalence relation can be realized as the path-connectedness relation on a Polish space, in fact a compact subset of $\mathbb{R}^3$. This is stated as \cite[Theorem 4.1]{becker1998number} with hints on the proof. Below we present a somewhat more detailed proof since our main construction is based on Becker's.

If $X$ is a Polish space, then the collection $\mathcal{K}(X)$ of compact subsets of $X$ is also a Polish space under Hausdorff distance, or equivalently under Vietoris topology. The following facts are standard.

\begin{lemma}\label{lemma_compact union}
Let $X$ be Polish. 

(i) If $\mathcal{F}\subseteq\mathcal{K}(X)$ is compact then $\bigcup\mathcal{F}$ is compact.

(ii) If $Z$ is compact and $f:Z\rightarrow\mathcal{K}(X)$ is continuous, then $\bigcup_{z\in Z}\{z\}\times f(z)$ is compact.
\end{lemma}

\begin{theorem}[Becker]\label{Becker}
If $E$ is any analytic equivalence relation on the middle-third Cantor set $\mathcal{C}$, then there is a compact subset $K\subseteq\mathbb
R^3$ such that the path-connectedness relation on $K$ is Borel bireducible with $E$.
\end{theorem}
\begin{proof}
We identify $\mathcal{C}$ with ${}^{\omega}2$. Theorems 33.17 and 37.11 of \cite{kechris2012classical} describe a continuous map sending a nonempty tree $T\subseteq{}^{<\omega}\omega$ to a compact $K_T\subseteq\mathbb{R}^2$ with the following properties:

\begin{enumerate}
    \item $K_T$ is a nonempty compact subset of $[0,1]\times[0,1]$;

    \item $K_T$ has at most two path-connected components, represented by the points $(0,1)$ and $(0,0)$ respectively;

    \item $K_T$ is path-connected if and only if $T$ has a branch.
\end{enumerate}

\begin{figure}
\centering

\begin{tikzpicture}[
scale=10]
  % axes (optional)
  \draw[-] (0,0) -- (1,0) node[midway,below] {$l$};
  \draw[-] (1,0) -- (1,1) node[midway,right] {$l_\emptyset$};
    \node at (0,1) [circle,fill,inner sep=1.8pt,label=left:{$r_\emptyset$}]{} ;

%1st level zigzag start

  \draw
    (0,1)   % first point
    \foreach \k in {0,1,2,...,5} {
      -- ({0.14*(15*(2*\k+1)/(15+(2*\k+1)))},{0.85+0.15*(-1)}) node[circle,left] {$\k$} 
      -- ({0.14*(15*(2*\k+2)/(15+(2*\k+2)))},{0.85+0.15})
    };

    \draw ({0.14*(15*(2*5+2)/(15+(2*5+2)))},{0.85+0.15}) 
    -- ({0.14*(15*(2*6+1)/(15+(2*6+1)))},{0.85+0.15*(-1)}) node[left] {$6$}
    -- (0.98,0.78);

    \foreach \k in {0,1,2,...,6} {
     \node at ({0.14*(15*(2*\k+1)/(15+(2*\k+1)))},{0.85+0.15*(-1)}) [circle,fill,inner sep=1.8pt]{};
    };

%1st level zigzag end

 \foreach \k in {0,2,5}{
\draw ({0.14*(15*(2*\k+1)/(15+(2*\k+1)))},{0.85+0.15*(-1)})
-- ({0.14*(15*(2*\k+1)/(15+(2*\k+1)))},0.5) node[circle,left] {$r_{(\k)}$} ;
\node at ({0.14*(15*(2*\k+1)/(15+(2*\k+1)))},0.5)[circle,fill,inner sep=1.8pt]{};
 }

\draw (0.33,0.5) -- (0.33,0) node[pos=0.75,left]{$l_{(0)}$};

\draw (0.66,0.5) -- (0.66,0) node[pos=0.75,left]{$l_{(2)}$};

\draw (0.97,0.5) -- (0.97,0) node[pos=0.75,left]{$l_{(5)}$};

%r_0 zigzag

  \draw
    (0.14*15/16,0.5)   % first point
    \foreach \k in {0,1,2} {
      -- ({0.12*15/16+0.04*(15*(2*\k+1)/(15+(2*\k+1)))},{0.4+0.1*(-1)}) 
      -- ({0.12*15/16+0.04*(15*(2*\k+2)/(15+(2*\k+2)))},{0.4+0.1})
    };

    \draw ({0.12*15/16+0.04*(15*(2*2+2)/(15+(2*2+2)))},{0.4+0.1}) 
    -- ({0.12*15/16+0.04*(15*(2*3+1)/(15+(2*3+1)))},{0.4+0.1*(-1)}) 
    -- (0.31,0.4);

    \foreach \k in {0,1,2,3} {
     \node at ({0.12*15/16+0.04*(15*(2*\k+1)/(15+(2*\k+1)))},{0.4+0.1*(-1)}) [circle,fill,inner sep=1.4pt]{};
    };

\draw ({0.12*15/16+0.04*(15*(2*1+1)/(15+(2*1+1)))},{0.4+0.1*(-1)}) -- ({0.12*15/16+0.04*(15*(2*1+1)/(15+(2*1+1)))},{0.2}) node[left]{$r_{(01)}$};

\node at ({0.12*15/16+0.04*(15*(2*1+1)/(15+(2*1+1)))},{0.2})[circle,fill,inner sep=1.4pt]{};

\draw ({0.12*15/16+0.04*(15*(2*3+1)/(15+(2*3+1)))},{0.4+0.1*(-1)}) -- ({0.12*15/16+0.04*(15*(2*3+1)/(15+(2*3+1)))},{0.2}) node[left]{$r_{(03)}$};

\node at ({0.12*15/16+0.04*(15*(2*3+1)/(15+(2*3+1)))},{0.2})[circle,fill,inner sep=1.4pt]{};

%r_2 zigzag

  \draw
    (0.14*15/4,0.5)   % first point
    \foreach \k in {0,1,2} {
      -- ({0.14*15/4+0.025*(15*(2*\k+1)/(15+(2*\k+1)))},{0.4+0.1*(-1)}) 
      -- ({0.14*15/4+0.025*(15*(2*\k+2)/(15+(2*\k+2)))},{0.4+0.1})
    };

    \draw ({0.14*15/4+0.025*(15*(2*2+2)/(15+(2*2+2)))},{0.4+0.1}) 
    -- ({0.14*15/4+0.025*(15*(2*3+1)/(15+(2*3+1)))},{0.4+0.1*(-1)}) 
    -- (0.65,0.4);

    \foreach \k in {0,1,2,3} {
     \node at ({0.14*15/4+0.025*(15*(2*\k+1)/(15+(2*\k+1)))},{0.4+0.1*(-1)}) [circle,fill,inner sep=1.4pt]{};
    };

\draw ({0.14*15/4+0.025*(15*(2*0+1)/(15+(2*0+1)))},{0.4+0.1*(-1)}) -- ({0.14*15/4+0.025*(15*(2*0+1)/(15+(2*0+1)))},{0.2}) node[left]{$r_{(20)}$};

\node at ({0.14*15/4+0.025*(15*(2*0+1)/(15+(2*0+1)))},{0.2})[circle,fill,inner sep=1.4pt]{};

%r_5 zigzag

  \draw
    (0.14*165/26,0.5)   % first point
    \foreach \k in {0,1} {
      -- ({0.14*165/26+0.017*(15*(2*\k+1)/(15+(2*\k+1)))},{0.4+0.1*(-1)}) 
      -- ({0.14*165/26+0.017*(15*(2*\k+2)/(15+(2*\k+2)))},{0.4+0.1})
    };

    \draw ({0.14*165/26+0.017*(15*(2*1+2)/(15+(2*1+2)))},{0.4+0.1}) 
    -- ({0.14*165/26+0.017*(15*(2*2+1)/(15+(2*2+1)))},{0.4+0.1*(-1)}) 
    -- (0.96,0.4);

    \foreach \k in {0,1,2} {
     \node at ({0.14*165/26+0.017*(15*(2*\k+1)/(15+(2*\k+1)))},{0.4+0.1*(-1)}) [circle,fill,inner sep=1.4pt]{};
    };

\draw ({0.14*165/26+0.017*(15*(2*2+1)/(15+(2*2+1)))},{0.4+0.1*(-1)}) -- ({0.14*165/26+0.017*(15*(2*2+1)/(15+(2*2+1)))},{0.2}) node[left]{$r_{(52)}$};

\node at ({0.14*165/26+0.017*(15*(2*2+1)/(15+(2*2+1)))},{0.2})[circle,fill,inner sep=1.4pt]{};

\end{tikzpicture}

	\caption{}
	\label{fig:slice}
\end{figure}

Figure \ref{fig:slice} shows the initial steps of the construction of $K_T$ for a typical $T$. It is copied from \cite[Figure 33.7]{kechris2012classical}, except that the left segment $p$ is deleted. The idea is that for each $s\in T$, we add a zig-zag version of topologist's sine curve starting at $r_s$, accumulating on the vertical segment $l_s$.

For $c_0,c_1\in {}^{\omega}2$, let $c_0\oplus c_1$ be the string with $c_0$ on even digits and $c_1$ on odd digits. Also define $p,q:{}^{\omega}2\rightarrow{}^{\omega}2$ by $p(c)(n)=c(2n)$ and $q(c)(n)=c(2n+1)$. By the identification we may view these as maps on $\mathcal{C}$. Let $A=\{c_0\oplus c_1:c_0Ec_1\}$, so $c\in A\Leftrightarrow p(c)Eq(c)$.

Since $A$ is analytic, there is a tree $T$ on $2\times\omega$ such that $c\in A\Leftrightarrow T_c$ has a branch. Let $K_0$ be $\bigcup_{c\in\mathcal{C}}K_{T_c}\times\{c\}$. Since the maps $c\mapsto T_c$ and $T\mapsto K_T$ are both continuous, $K_0$ is compact by Lemma \ref{lemma_compact union}. We have

\begin{center}
$c\in A\Leftrightarrow T_c$ has a branch $\Leftrightarrow$ $(0,0,c)$ and $(0,1,c)$ are path-connected in $K_0$.
\end{center}

\begin{figure}
    \centering

\tdplotsetmaincoords{78}{-35}
\begin{tikzpicture}
		[scale=4,
        tdplot_main_coords,
			cube/.style={very thick,black},
            cantor/.style={very thick,black,dashed},
			grid/.style={very thin,gray},
			axis/.style={->,blue,thin}]

	\draw[axis] (-1.5,0,0) -- (1.5,0,0) node[anchor=west]{$x$};
	\draw[axis] (0,-0.3,0) -- (0,2,0) node[anchor=east]{$y$};
	\draw[axis] (0,0,-0.3) -- (0,0,2) node[anchor=west]{$z$};

	\draw[cube] (0,0,0) -- (1,0,0) -- (1,1,0) ;
    \draw[cube] (0,0,0.66) -- (1,0,0.66) -- (1,1,0.66) ;
	\draw[cube] (0,0,1) -- (1,0,1) -- (1,1,1) ;

	\draw[cantor] (0,0,0) -- (0,0,1);
	\draw[cantor] (0,1,0) -- (0,1,1);
	\draw[cantor] (1,0,0) -- (1,0,1);
	\draw[cantor] (1,1,0) -- (1,1,1);

    \draw[cube]
    (0,1,1)   % first point
    \foreach \k in {1,...,10} {
      -- ({1-1/(\k+1)},{0.75+0.25*(-1)^(\k)},1)
    };

    \draw[cube]
    (0,1,0)   % first point
    \foreach \k in {1,...,10} {
      -- ({1-1/(\k+1)},{0.75+0.25*(-1)^(\k)},0)
    };

    \draw[cube]
    (0,1,0.66)   % first point
    \foreach \k in {1,...,10} {
      -- ({1-1/(\k+1)},{0.75+0.25*(-1)^(\k)},0.66)
    };

    \draw[cantor] (-0.5,0.5,0) -- (-1.5,0.5,0);

    \draw (0,0,0.66) node[anchor=east]{$(0,0,c)$};

    \draw (0,1,0.66) node[anchor=east]{$(0,1,c)$};

    \draw[cube] (0,1,0.66) -- (-0.8,0.5,0) node[anchor=north]{$q(c)'$};

    \draw[cube] (0,0,0.66) -- (-1.2,0.5,0) node[anchor=north]{$p(c)'$};
    
\end{tikzpicture}

    \caption{}
    \label{fig:K}
\end{figure}

For $c\in\mathcal{C}$ let $c'=(-0.5-c,0.5,0)$, and let $\mathcal{C}'=\{c':c\in\mathcal{C}\}$. Add segments connecting each $(0,0,c)$ to $p(c)'$, and call these type I segments; similarly add segments connecting each $(0,1,c)$ to $q(c)'$, and call these type II segments. Finally let $K$ be the union of $K_0$ and these segments, as shown in Figure \ref{fig:K}, where we use dashed lines to indicate Cantor sets. Again it can be checked to be compact by Lemma \ref{lemma_compact union}.

Note that different type I/II segments (such as the two shown in Figure \ref{fig:K}) do not intersect, except possibly at their common endpoint on $\mathcal{C}'$. If the segments have different types this is because they are on different sides of the plane $y=0.5$; if they have the same type then this is because $\mathcal{C}'$ and $(0,0)\times\mathcal{C}$ lie in skew lines, similarly for $(0,1)\times\mathcal{C}$. 

This finishes the construction of $K$. We now want to show that $\sim_K$ is bireducible with $E$. To this end we argue that:

(1) for $c_0,c_1\in\mathcal{C}$, $c_0Ec_1$ if and only if $c_0'$ is path-connected to $c_1'$ in $K$;

(2) there is a Borel map sending each $x\in K$ to some point in $\mathcal{C}'$ in the same path-connected component.

Note that without (2), all we can conclude is that $E$ is isomorphic to a complete section of $\sim_K$, which may have lower Borel complexity than $\sim_K$ in general.

The following claim allows us in the proof of (1) to pretend that only finitely many type I/II segments are present.

\begin{claim}\label{claim_finiteness_1}
Divide each type I/II segment evenly into three parts. If $f:[0,1]\rightarrow K$ is continuous, then its image can only intersect the middle third of finitely many type I/II segments.
\end{claim}
\begin{proof}
If $0\leq s<t\leq1$ and $f(s),f(t)$ belong to two different type I/II segments, then there must exist $r\in[s,t]$ where $f(r)$ is either in $\mathcal{C}'$ or in $K_0$. Now suppose there is a sequence $(s_i:i<\omega)$ such that $f(s_i),\ i<\omega$ belong to the middle third of different segments; using the compactness of $[0,1]$ and passing to a subsequence, we may assume $(s_i:i<\omega)$ has a limit $s$. Then clearly $f$ is discontinous at $s$, a contradiction. 
\end{proof}

We can now finish the proof by showing (1) and (2).

(1) The forward direction is clear from construction. For the backward direction, suppose $f:[0,1]\rightarrow K$ is a path connecting $c_0'$ with $c_1'$. By the claim it only reachs the middle third of finitely many segments of type I/II. For a type I segment, say from $p(c)'$ to $(0,0,c)$, we define a continuous surjective map from the segment to itself as follows: contracts the first third onto $p(c)'$ and the last third onto $(0,0,c)$, and uniformly stretch the middle third. Also define a similar map on each type II segment. Putting these maps together we have a map $\varphi:K\rightarrow K$. Then $\varphi\circ f$ is still a path connecting $c_0'$ with $c_1'$, and it touches the interior of only finitely many type I/II segments. Then the path touches only finitely many $c'$, which must clearly be $E$-related.

(2) $K$ consists of three kinds of points: $\mathcal{C}'$, the interior points of type I/II segments, and $K_0=\bigcup_{c\in\mathcal{C}}K_{T_c}\times\{c\}$. We define a map that sends each $x\in K$ to some $c'\in\mathcal{C}'$ in the same path-component as follows. If $x\in \mathcal{C}'$ then send it to itself. If $x$ is an interior point of a segment of either type, send it to the endpoint of that segment on $\mathcal{C}'$. Now suppose $x\in K_{T_c}\times\{c\}$; referencing Figure \ref{fig:slice}, if $x$ is on the bottom side $l$ or on one of the vertical segments $l_s$, send it to $p(c)'$; otherwise, $x$ is on the ``zigzag'' part and we send it to $q(c)'$. This is easily seen to be a Borel map.
\end{proof}

In contrast to the situation in $\mathbb{R}^3$, the path-connectedness relation on a Polish $X\subseteq\mathbb{R}^2$ is either smooth or hyperfinite \cite{uyar2025complexity}.

We now define, from an analytic equivalence relation $E$ on a Polish space $X$, a new analytic equivalence relation $F(E)$ on the free group $F(X)$, and show that $F(E)$ can be realized by homotopy of loops. The definition of $F(E)$ is inspired by the proof of \cite[Corollary 4.2]{becker1998number}.

Let $W(X)={}^{^{<\omega}}(X\times\{-1,1\})$ be the Polish space of all words in $X$, viewed as the disjoint union of the Polish spaces ${}^n(X\times\{-1,1\})$. For ease of notation, if $x\in X$ then denote $(x,1)$ by $x$ and $(x,-1)$ by $x^{-1}$. Let $F(X)\subseteq W(X)$ be the free group on $X$, namely the subset consisting of all reduced words. Clearly $F(X)$ is open in $W(X)$ and hence Polish, but we are \textit{not} claiming it is a Polish group, although the group operation can be seen to be Borel using a case-by-case analysis.

Let $N_E$ be the normal subgroup of $F(X)$ generated by $\{x^{-1}y:x,y\in X,\ xEy\}$. 

\begin{definition}
$F(E)$ is the coset equivalence relation on $F(X)$ with respective to $N_E$.
\end{definition}

Namely, $uF(E)v\Leftrightarrow uN_E=vN_E\Leftrightarrow u^{-1}v\in N_E$ for $u,v\in F(X)$. Since $F(X)$ is a Borel group and $N_E$ is clearly analytic, so is $F(E)$. The quotient $F(X)/F(E)$ is naturally identified with the free group on $X/E$, so we have the suggestive formula $F(X)/F(E)\simeq F(X/E)$.

\begin{theorem}\label{theorem_realization}
For any analytic equivalence relation $E$ on $\mathcal{C}$, there is a compact path-connected $L\subseteq\mathbb{R}^5$ such that $\sim_L$ is bireducible with $F(E)$.
\end{theorem}
\begin{proof}
We first construct such an $L$ in $\mathbb{R}^6$. Let $K\subseteq\mathbb{R}^3$ be as in Becker's theorem, so the path-connectedness relation on $K$ is bireducible with $E$. A natural option is $K\times S^1$, but there are several issues: this is not path-connected; it has more nontrivial loops than intended; if $c_0Ec_1$ then the corresponding loops are only free-homotopic and not homotopic. These can be solved by attaching a cone on $K$ to $K\times S^1$. To be concrete, we first choose a base point of $S^1$, say $a=(1,0)$. Consider $K\times S^1\times\{0\}$ in $\mathbb{R}^6$, and for each $x\in K$ add the segment between $(x,a,0)$ and $b:=(0,0,0,0,0,1)$. Let $L$ be the union of $K\times S^1\times\{0\}$ and these segments, with $b$ as the natural basepoint. Since every point in $K\times S^1\times\{0\}$ is path-connected to some point in $K\times\{a\}\times\{0\}$, which is in turn path-connected to $b$, the set $L$ is path-connected, and easily seen to be compact. Note that the cone also makes all loops in $K$ null-homotopic.

This finishes the definition of $L$. It is not hard to see that $K$ has covering dimension $1$, and that $L$ has covering dimension $2$. Therefore $L$ embeds into $\mathbb{R}^5$, cf. \cite[Theorem 50.5]{munkres2000topology}. It remains to show that $\sim_L$ is bireducible with $F(E)$.

For each $c\in\mathcal{C}$, let $g_c$ be the natural loop in $L$ starting at $b$, then goes to $(c',a,0)$, traverse the circle $\{c'\}\times S^1\times\{0\}$ once (in a fixed orientation) and goes back to $b$. Because of the cone added on $K$, if $c_0Ec_1$ then $g_{c_0}$ and $g_{c_1}$ are homotopic. To show that $\sim_L$ is bireducible with $F(E)$, it suffices to show that: 

(1) there is a Borel map sending an arbitrary loop in $\Omega L$ to a concatenation of powers of finitely many $g_{c}$'s, so in particular the $g_{c}$'s generate $\pi(L)$;

(2) if $c_0,\cdots,c_n$ are mutually $E$-unrelated, then $g_{c_0},\dots,g_{c_n}$ generate a free group in $\pi(L)$.

The arguments are elementary, albeit cumbersome, so we will omit some details. The groupoid version of Seifert-Van Kampen theorem \cite{brown2006topology} can be used to give a conceptually clearer proof of (2). For brevity let $CK$ denote the cone on $K$, namely the union of $K\times\{a\}\times\{0\}$ and the part of $L$ with positive last coordinate.

(1) Suppose $f:[0,1]\rightarrow L$ is continuous and $f(0)=f(1)=b$. The set $f^{-1}(L\setminus CK)$ is open, hence a countable disjoint union of open intervals. On the closure $[s_0,s_1]$ of each of these intervals, $f$ can be identified with the product of two maps $g:[s_0,s_1]\rightarrow K$ and $h:[s_0,s_1]\rightarrow S^1$, where $h(s_0)=h(s_1)=a$. Note that $h$ is null-homotopic unless it touches the ``left half'' of $S^1$, which can happen only for finitely many intervals $[s_0,s_1]$ by an argument similar to Claim \ref{claim_finiteness_1}. Thus after a homotopy that ``collapses the left half of $S^1$'', we can make $h$ constantly $a$ (so $f$ has image in $K\times\{a\}\times\{0\}$) on all but finitely many intervals. On each of the exceptional interval $[s_0,s_1]$, we first ``slide $g$ and $h$ apart'', more precisely we homotope $g$ to be constantly $x:=g(s_0)$ on the first half and $h$ to be constantly $a$ on the second half. Then on the first half we homotope $f$ to be the segment from $f(s_0)$ to $b$ followed by suitable power of $g_c$, and then followed by the segment from $b$ to $f(\frac{s_0+s_1}{2})$; here $c\in\mathcal{C}$ is the canonical point (see the proof of Theorem \ref{Becker}) such that $c'$ and $x$ lie in the same path-component of $K$.

To summarize, we have homotoped $f$ so that there are finitely many intervals on which $f$ is a power of some $g_c$, and on the complement intervals $f$ has image in $CK$. Now on each complement interval we ``push up'' $f$ to the basepoint $b$, and then shrink the interval to a point; the result is a concatenation of powers of finitely many $g_{c}$'s as desired.

(2) Suppose towards a contradiction that $f$ is a concatenation of nonzero powers of $g_{c_0},\dots,g_{c_n}$ where $c_i$'s are mutually non-equivalent, and $H$ is an homotopy between $f$ and the constant map at $b$. Again by an argument similar to Claim \ref{claim_finiteness_1}, $\im(H)$ can intersect only finitely many path-components of $K\times S^1\times\{a\}$, and each intersection is closed. So there are finitely many path-components $P_1,\dots,P_m$ of $K$ and compact subsets $K_i\subseteq P_i$ such that $\im(H)\subseteq CK\cup\bigcup_{i=1}^n K_i\times S^1\times\{0\}=:L'$; we may assume each $c_i'$ is contained in some $P_j$. We have that $f$ is null-homotopic in $L'$.

Consider the quotient map $\pi$ on $L'$ that collapses $K_i$ to a point, and thus $K_i\times S^1$ to a circle, and also collapses $CK$ to a point. Since $K_i$ are disjoint compact subsets, the quotient space can be seen to be a wedge sum of $m$ circles, whose fundamental group is free on $m$ generators. Moreover, $\pi$ sends each $g_{c_i}$ to a generator, and thus it sends $f$ to a nontrivial element, contradicting that $f$ is null-homotopic in $L'$.
\end{proof}

Thus to show that an equivalence relation $E$ is Borel bireducible with some $\sim_X$, it suffices if $E$ is Borel bireducible with $F(E)$. We give such examples in the next section.

\begin{question}
Can the set $L$ in Theorem \ref{theorem_realization} be a compact subset of $\mathbb{R}^3$ or $\mathbb{R}^4$? Can it be locally connected?
\end{question}

If $L$ is compact, then by Hahn--Mazurkiewicz Theorem, being locally connected is the same as being locally path-connected. Note that the approach via Becker's Theorem does not seem to work.

\begin{proposition}
If $X$ is a locally path-connected Polish space, then it has at most countably many path-components. Thus the path-connectedness relation on $X$ is smooth.
\end{proposition}
\begin{proof}
Otherwise, there exists an uncountable set $A$ of points that belong to different path-components. Since $X$ is second countable, $A$ has an accumulation point in itself, namely there is $x\in A$ such that any neighborhood of $x$ contains infinitely many other $y\in A$. But $x$ also has a path-connected neighborhood, a contradiction. 
\end{proof}

\section{Examples}

Let $E$ be an analytic equivalence relation on $X$. We first establish some basic properties of the operation $E\mapsto F(E)$. 

\begin{lemma}\label{F(E)_above_E}
$E\sqsubseteq_c F(E)$.
\end{lemma}
\begin{proof}
Send an $x\in X$ to the word $x$.
\end{proof}

Let $F'(X)\subseteq F(X)$ consist of words without subword of form $xy^{-1}$ or $x^{-1}y$ where $xEy$. Denote the restriction of $F(E)$ to $F'(X)$ by $F'(E)$

\begin{lemma}
If $E$ is Borel then so are $F'(X)$ and $F(E)$. Moreover, $F'(E)$ is bireducible with $F(E)$.
\end{lemma}
\begin{proof}
Clearly $F'(X)$ and $F'(E)$ are Borel: two words in $F'(X)$ are equivalent just in case they have the same length and are pointwise equivalent. Also, $F'(X)$ is an $F(E)$-complete section, and there is a Borel map from $F(X)$ to $F'(X)$ that sends a word to another word in the same $F(E)$-class: if there exist subwords like $xy^{-1}$ or $x^{-1}y$ where $xEy$, just delete the leftmost one, and repeat this until there is none. This shows bireducibility, and in particular $F(E)$ is Borel.
\end{proof}

\begin{lemma}\label{F(E)_preservation}
For each of the following properties, if $E$ has it then so does $F'(E)$.

(i) smooth; (ii) hyperfinite; (iii) countable; (iv) hypersmooth.
\end{lemma}
\begin{proof}
Consider $F=\bigsqcup_{n<\omega}(E\times 2)^n$, which can be viewed as an equivalence relation on $W(X)$. Then $F(E)$ restricted to $F'(X)$ coincides with $F$ restricted to $F'(X)$. Clearly if $E$ has one of these properties then so does $F$, and these properties are hereditary.
\end{proof}

\begin{theorem}
For each of the following analytic equivalence relation $E$, there exists a path-connected Polish space $X$ such that $E$ is Borel bireducible with $\sim_X$.

(i) $E_0$; (ii) $E_\infty$; (iii) $E_1$; (iv) the universal analytic equivalence relation.
\end{theorem}
\begin{proof}
(i) $F(E_0)$ is essentially hyperfinite by Lemma \ref{F(E)_preservation} and non-smooth by Lemma \ref{F(E)_above_E}. It is well-known that $E_0$ is up to Borel bireducibility the unique non-smooth hyperfinite equivalence relation.

(ii) $F(E_\infty)$ is essentially countable by Lemma \ref{F(E)_preservation} and above $E_\infty$ by Lemma \ref{F(E)_above_E}.

The arguments for (iii) and (iv) are similar.
\end{proof}

\begin{theorem}
There exists a path-connected Polish space $X$ such that $\sim_X$ is smooth but not faithfully smooth.
\end{theorem}
\begin{proof}
Take $E$ to be a smooth but not faithfully smooth equivalence relation; for example, consider a closed $F\subseteq\mathcal{N}^2$ whose projection onto $\mathcal{N}$ is non-Borel, and define an equivalence relation $E$ on $F$ by $(x,y)E(x',y')\Leftrightarrow x=x'$. Since $E\sqsubseteq_c F(E)$, in particular $E$ faithfully reduces to $F(E)$, so $F(E)$ cannot be faithfully smooth, because faithful reductions are closed under composition.
\end{proof}

Finally we discuss the complexity of $\sim_X$ for some well-known spaces. The harmonic archipelago, denoted $\mathbb{HA}$, is obtained from the disk $(x-1)^2+y^2\leq 1$ as follows: for each $n\geq 1$, replace a disk between the $n$-th circle and the $n+1$-th circle of the earring $\mathbb{E}$ by the cone on the boundary of that disk with unit height (imagine that we push up the disk to form a hill).

Recall that $\pi(\mathbb{E})$ is isomorphic to a Borel subgroup of the Polish group $\varprojlim_n F_n$, where $F_n=\langle g_1,\dots,g_n\rangle$ is the free group on $n$ generators; we identify them from now on. We also identify $g_n$ with the sequence $(\dots,g_n,g_n,1,\dots,1)$. Now $\mathbb{E}$ is a subspace of $\mathbb{HA}$, and it is not hard to see that every loop in $\mathbb{HA}$ is homotopic to one in $\mathbb{E}$, so the natural map $\theta:\pi(\mathbb{E})\rightarrow\pi(\mathbb{HA})$ is surjective. Its kernel is characterized by \cite[Theorem 7]{fabel2005fundamental}:

\begin{theorem*}
A sequence $(w_n:n<\omega)\in\pi(\mathbb{E})$ belongs to $\ker(\theta)$ if and only if there exists $N$ such that for every $n\geq N$, if we replace each occurrence of $g_i,\ i< N$ in $w_n$ by $g_N$, the resulting (possibly non-reduced) word is equivalent to the empty word.
\end{theorem*}

Intuitively, the hills make all the $g_n$'s homotopic, but each particular homotopy can only use finitely many hills.

\begin{theorem}
(i) $\sim_{\mathbb{HA}}$ is Borel.

(ii) $E_1$ Borel reduces to $\sim_{\mathbb{HA}}$.
\end{theorem}
\begin{proof}
(i) Since two loops $f,g\in\Omega\mathbb{E}$ are homotopic in $\mathbb{HA}$ if and only if $[f]\ker(\theta)=[g]\ker(\theta)$ if and only if $[f^{-1}\cdot g]\in \ker(\theta)$, it suffices to show that $\ker(\theta)$ is a Borel subset of $\pi(\mathbb{E})$. This is clear from the above theorem, since each $F_n$ is countable discrete.

(ii) Fix a partition of the positive integers into infinitely many infinite sets $A^p,\ p<\omega$. Let $a_p,b_p$ be the first two elements of $A^p$. The map sending each $c\in{}^\omega 2$ to $\{c\upharpoonright n:n<\omega\}\in\mathcal{P}({}^{<\omega}2)$ is continuous and its range is an almost disjoint family;  for each $p$ we can translate this to a continuous almost disjoint family on $A^p$, denoted $\{A^p_c:c\in{}^\omega 2\}$. Let $f^p_c$ be the loop that traverses each circle $(x-\frac{1}{n})^2+y^2=(\frac{1}{n})^2$ counterclockwise once for $n\in A^p_c$.

Define a map $\eta:{}^{\omega}({}^\omega 2)\rightarrow\Omega\mathbb{HA}$ that sends $(c_i:i<\omega)$ to the infinite concatenation of the loops $f^p_{c_p}g_{a_p}g_{b_p}^{-1}(f^p_{c_p})^{-1},\ 1\leq p<\omega$. If two sequences $\vec{c}_1,\vec{c}_2\in{}^{\omega}({}^\omega 2)$ are $E_1$-related, then since $g_{a_p}g_{b_p}^{-1}$ is null-homotopic, $\eta(\vec{c}_1)$ and $\eta(\vec{c}_2)$ are homotopic. Conversely, if $\vec{c}_1,\vec{c}_2\in{}^{\omega}({}^\omega 2)$ are $E_1$-unrelated, their images can be seen to be non-homotopic by the above theorem.
\end{proof}

The Griffiths space $\mathbb{GS}$, also known as the twin cone, is obtained by taking two copies of the cone over $\mathbb{E}$ and identifying their basepoints. By similar arguments, $\sim_{\mathbb{GS}}$ can be shown to be Borel and above $E_1$. The fundamental groups of $\mathbb{HA}$ and $\mathbb{GS}$ have recently been shown to be isomorphic \cite{corson2025double}, answering a long-standing question in the area. However, the proof involves a transfinite induction of length continuum and lots of arbitrary choices, and thus is highly non-constructive.

\begin{question}
Are $\sim_\mathbb{HA}$ and $\sim_\mathbb{GS}$ Borel bireducible?
\end{question}

\bibliographystyle{plain}
\bibliography{ref}

\end{document}